\documentclass[10pt]{amsart}

\usepackage{amsmath,amssymb,amsthm,mathtools}
\usepackage[T1]{fontenc}
\usepackage{lmodern}
\usepackage[margin=1in]{geometry}
\usepackage{enumitem}
\usepackage{microtype}
\usepackage[hidelinks]{hyperref}
\usepackage{graphicx}
\usepackage{color}
\definecolor{wine}{rgb}{0.5,0,0.2}
\usepackage{hyperref}
\hypersetup{colorlinks=true, linkcolor=blue, citecolor=magenta, urlcolor=wine}
\usepackage{url}
\usepackage{fancyhdr}
\usepackage{booktabs}
\usepackage{tabularx}

\newtheorem{theorem}{Theorem}[section]
\newtheorem{lemma}[theorem]{Lemma}
\newtheorem{corollary}[theorem]{Corollary}
\newtheorem{prop}[theorem]{Proposition}
\newcommand{\mathsc}[1]{\textnormal{\textsc{#1}}}
\makeatletter
\def\th@claimsc{%
    \th@plain
    \thm@headfont{\normalfont}%
    \def\thmhead##1##2##3{%
        \mathsc{\thmhead@plain{##1}{##2}{##3}}}%
}
\makeatother
\theoremstyle{claimsc}
\newtheorem{claim}{Claim}

\theoremstyle{definition}
\newtheorem{example}[theorem]{Example}
\newtheorem{question}[theorem]{Question}

\newcommand{\nn}{\mathbb{N}}
\newcommand{\zz}{\mathbb{Z}}
\newcommand{\qq}{\mathbb{Q}}
\newcommand{\rr}{\mathbb{R}}
\newcommand{\gp}{\mathcal{G}}
\newcommand{\A}{\mathcal{A}}
\newcommand{\F}{\mathsf{F}}
\DeclareMathOperator{\supp}{supp}
\DeclareMathOperator{\cone}{cone}



\begin{document}
	
\mbox{}

\title{The Bi-LF Positive Conjecture holds}

\author{Felix Gotti}
\address{Department of Mathematics\\MIT\\Cambridge, MA 02139}
\email{fgotti@mit.edu}

\author{Omar Graia}
\address{CrowdMath\\Cambridge, MA 02139}
\email{omar.graia2@gmail.com}

\author{Neil Panse}
\address{CrowdMath\\Cambridge, MA 02139}
\email{neilpanse@gmail.com}

\date{\today}

\begin{abstract}
	We prove that a semidomain has length-factorial additive and multiplicative monoids if and only if it is isomorphic to the semiring of nonnegative integers or is a unique factorization domain. Consequently, bi-length-factoriality and bi-factoriality coincide for all semidomains, settling the Bi-LF Positive Conjecture. The main step shows that additive length-factoriality and multiplicative factoriality force additive factoriality, without an additive reducedness assumption. Multiplicative factoriality forces every semidomain with a nonzero additive unit to be a unique factorization domain. In the additively reduced case, a uniform multiplier for lattice points in a finitely generated cone yields cancellation of divisibility across sums, which leads to an obstruction to additive atoms dividing positive integers and completes the structural argument. We also consider related questions on quasi-factoriality and classify the rational and transcendental monogenic positive semirings whose multiplicative monoids are quasi-$n$-factorial for a prescribed integer $n \ge 2$.
\end{abstract}

\bigskip
\maketitle

\bigskip
%%%%%%%%%%%%%%%%%%%%%%
%%%%%%%%%%%%%%%%%%%%%%
\section{Introduction}
\label{sec:intro}

We let $\nn$, $\zz$, $\qq$, and $\rr$ denote the sets of positive integers, integers, rational numbers, and real numbers, respectively, and set $\nn_0 := \nn \cup \{0\}$. Throughout this paper, monoids are commutative and cancellative.
\smallskip

An atomic monoid is \emph{factorial} if each nonunit has a unique factorization into atoms, up to order and associates. It is \emph{length-factorial} if distinct factorizations of the same element have distinct lengths. Thus, factoriality implies length-factoriality. We collect the factorization conventions and other preliminaries in Section~\ref{sec:background}.
\smallskip

Length-factoriality was introduced by Coykendall and Smith in their 2011 paper \emph{On unique factorization domains}~\cite{CS11} under the name \emph{other-half-factoriality}. It complements half-factoriality: an atomic monoid is \emph{half-factorial} if all factorizations of each element have the same length, whereas length-factoriality requires factorizations of the same element and the same length to coincide up to order and associates. Thus, a monoid is factorial if and only if it is both half-factorial and length-factorial. Coykendall and Smith proved that an integral domain is other-half-factorial if and only if it is a unique factorization domain~\cite[Corollary 2.11]{CS11}; atomicity is included in their definition of other-half-factoriality. They also exhibited monoids that are length-factorial but not factorial~\cite[Example 2.13]{CS11}, showing that this characterization does not extend to arbitrary monoids. The term \emph{length-factoriality} was adopted by Chapman, Coykendall, Gotti, and Smith in their 2021 paper \emph{Length-factoriality in commutative monoids and integral domains}~\cite{CCGS21}, which developed the theory in these settings and explained the terminology in~\cite[Remark 2.2]{CCGS21}.
\smallskip

A \emph{semidomain} in this paper is a subsemiring of a field containing $0$ and $1$. It is \emph{positive} if it is isomorphic to a subsemiring of $\rr_{\ge 0}$ containing $1$. For a semidomain $S$, we write $S^* := S \setminus \{0\}$ for its multiplicative monoid.
\smallskip

Baeth, Chapman, and Gotti~\cite{BCG21} initiated the simultaneous study of factorization in the additive and multiplicative monoids of positive semirings. Their Bi-UF Positive Conjecture~\cite[Conjecture 7.7]{BCG21} asserts that $\nn_0$ is the only positive semiring whose two monoids are factorial. The conjecture was verified in several classes. Deng, Gotti, and Zeng~\cite[Corollary 5.4]{DGZ26} proved the stronger bi-half-factorial assertion for positive rational monogenic semidomains. Gotti, Graia, Han, and Liang~\cite[Theorem 3.6 and Corollary 3.7]{GGHL26} established the quadratic monogenic case, and Bilakanti, Gotti, Kandasamy, Liang, Liu, Polo, Yang, and Yao~\cite[Theorem 5.7]{BGKLLPYY26} established the case of semirings generated by finitely many positive algebraic numbers.
\smallskip

Graia~\cite[Theorem 4.3]{oG26} proved the full Bi-UF Positive Conjecture. Panse~\cite{nP26} gives another proof and establishes the corresponding classification for additively reduced semidomains, without requiring positivity. Here \emph{additively reduced} means that $0$ is the only unit of the additive monoid. Related work of Gonzalez, Polo, and Rodriguez~\cite{GPR24} studies localization of unique factorization semidomains and provides an example outside the positive setting whose additive and multiplicative monoids are both half-factorial~\cite[Example 4.6]{GPR24}.
\smallskip

Bu, Vulakh, and Zhao extended the Coykendall--Smith characterization to semidomains: the multiplicative monoid of a semidomain is length-factorial if and only if it is factorial~\cite[Corollary 5.4]{BVZ23}. In the same paper, they formulated the \emph{Bi-LF Positive Conjecture}, asserting that a positive semidomain is bi-length-factorial if and only if it is bi-factorial~\cite[Conjecture 5.6]{BVZ23}. Here the prefix \emph{bi-} means that the property holds for both the additive and multiplicative monoids. The related question of whether $\nn_0$ is the only positive bi-length-factorial semidomain had been posed earlier by Baeth, Chapman, and Gotti~\cite[Question 7.8(2)]{BCG21} and was recalled by Bu, Vulakh, and Zhao~\cite[Question 5.5]{BVZ23}. In view of the Bi-UF Positive Theorem, the conjecture of Bu, Vulakh, and Zhao is equivalent to an affirmative answer to this question. We establish both assertions as consequences of the following classification of all bi-length-factorial semidomains.
\smallskip

\begin{theorem}[Classification of bi-length-factorial semidomains]\label{thm:bilf}
    Let $S$ be a semidomain. Then $(S, +)$ and $S^*$ are length-factorial if and only if $S$ is isomorphic to $\nn_0$ or is a unique factorization domain.
\end{theorem}
\smallskip

In particular, a semidomain is bi-length-factorial if and only if it is bi-factorial. Among additively reduced semidomains, Theorem~\ref{thm:bilf} leaves only $\nn_0$, up to isomorphism, because the additive monoid of a nonzero ring is not reduced. Every positive semidomain is additively reduced, so this includes the positive case.
\smallskip

By the semidomain characterization of Bu, Vulakh, and Zhao recalled above, the hypothesis that $S^*$ is length-factorial already implies that it is factorial. The main step is therefore the following additive upgrade.
\smallskip

\begin{theorem}[Additive upgrade]\label{thm:upgrade}
    Let $S$ be a semidomain. If $(S, +)$ is length-factorial and $S^*$ is factorial then $(S, +)$ is factorial.
\end{theorem}
\smallskip

Section~\ref{sec:background} introduces the factorization terminology and the group and cone constructions used in the proofs, and shows that a multiplicatively factorial semidomain with a nonzero additive unit is a unique factorization domain. In Section~\ref{sec:upgrade}, we reduce Theorem~\ref{thm:upgrade} to the additively reduced case, analyze the cyclic group of additive relations, and use a cone argument to establish cancellation of divisibility across sums. This allows us to determine the signs of certain coordinates of the relation, prove an integer-divisor obstruction, and derive a contradiction by considering multiplication on the additive factorization group. Section~\ref{sec:completion} treats the additively reduced bi-factorial case and completes the unrestricted classification in Theorem~\ref{thm:bilf}. In Section~\ref{sec:quasi-factoriality}, we discuss quasi-factoriality in integral domains and positive semirings, give a quasilocal criterion for a relation of length two, and treat rational and transcendental monogenic semirings.
\smallskip

%%%%%%%%%%%%%%%%%%%%%%%%%%%%%
%%%%%%%%%%%%%%%%%%%%%%%%%%%%%
\bigskip
\section{Background}\label{sec:background}

We collect the factorization terminology and the algebraic and geometric constructions used in the proofs.
\smallskip

\medskip
\subsection{Monoids and Factorizations}

A \emph{monoid} is a set equipped with an associative operation and an identity element. As stipulated above, all monoids in this paper are commutative and \emph{cancellative}: in multiplicative notation, $ab = ac$ implies $b = c$. An element of a monoid $M$ is a \emph{unit} if it has an inverse in $M$. The group of units is denoted by $M^\times$. The monoid is \emph{reduced} if its identity is its only unit. Two elements are \emph{associates} if one is a unit multiple of the other. For $a, b \in M$, we write $a \mid_M b$ if $b = ac$ for some $c \in M$, and say that $a$ \emph{divides} $b$.
\smallskip

A nonunit $a \in M$ is an \emph{atom}, or \emph{irreducible element}, if each equality $a = bc$ with $b, c \in M$ forces one factor to be a unit. The monoid $M$ is \emph{atomic} if every nonunit is a finite product of atoms. A \emph{factorization} of an element is its expression as a unit times a product of atoms, considered up to order and associates. Its \emph{length} is the number of atom factors, counted with multiplicity; a unit has the empty factorization, of length zero. An atomic monoid is \emph{factorial} if each element has exactly one factorization, \emph{length-factorial} if distinct factorizations of each element have distinct lengths, and \emph{half-factorial} if all factorizations of each element have the same length. A length-factorial monoid is \emph{proper} if it is not factorial. These definitions apply to additively written monoids by replacing products by sums and the identity $1$ by $0$. In particular, every abelian group, regarded as a monoid, is factorial because all its elements are units.
\smallskip

A nonunit $p \in M$ is \emph{prime} if $p \mid_M ab$ implies $p \mid_M a$ or $p \mid_M b$ for all $a, b \in M$. A \emph{greatest common divisor} of a nonempty finite family of elements is a common divisor that is divisible by every common divisor of that family. In a factorial monoid, every atom is prime, and greatest common divisors exist and are unique up to associates. These facts follow by comparing atom multiplicities in factorizations: a greatest common divisor takes the minimum multiplicity of each atom. The same comparison gives Euclid's lemma in the form used here: if a greatest common divisor of $a$ and $b$ is a unit and $a \mid_M bc$ then $a \mid_M c$.
\smallskip

Fix an integer $n \ge 2$. An atomic monoid $M$ is \emph{quasi-$n$-factorial} if each element has at most one factorization of length $n$, up to order and associates. Thus, $M$ is length-factorial if and only if it is quasi-$n$-factorial for every integer $n \ge 2$. An atomic monoid is \emph{square-factorial} if, for atoms $u, v, w$, the equality $u^2 = vw$ implies that $u, v, w$ are associates. These notions and their relation to unique factorization are discussed in~\cite{CC26}.
\smallskip

\medskip
\subsection{Semidomains and Integral Domains}

A \emph{semiring} is a set with associative and commutative addition, associative multiplication, additive identity $0$, and multiplicative identity $1$, such that multiplication distributes over addition and $0$ is absorbing for multiplication. A \emph{semidomain} in this paper is a subsemiring of a field containing $0$ and $1$. In particular, its addition and its multiplication on nonzero elements are commutative and cancellative. It is \emph{positive} if it is isomorphic to a subsemiring of $\rr_{\ge 0}$ containing $1$. We identify a positive semidomain with such a subsemiring when convenient.
\smallskip

For a semidomain $S$, the notation $S^* := S \setminus \{0\}$ denotes the multiplicative monoid, and $S^\times$ denotes its group of units. The sets of additive and multiplicative atoms are denoted by $\A_+(S)$ and $\A_\times(S)$, respectively. An element of their intersection is a \emph{common atom}. A semidomain is \emph{additively reduced} if $0$ is the only unit of $(S, +)$. Every positive semidomain is additively reduced. In an additively reduced semidomain, an additive atom is precisely a nonzero element that cannot be written as a sum of two nonzero elements. We say that $S$ is \emph{bi-atomic}, \emph{bi-factorial}, \emph{bi-length-factorial}, or \emph{bi-half-factorial} if both $(S, +)$ and $S^*$ have the corresponding property.
\smallskip

In an additively reduced semidomain $S$, a finite sum of nonzero elements is nonzero, since otherwise a summand would have an additive inverse in $S$. In particular, the ambient field has characteristic zero. We identify its prime subfield with $\qq$ and the canonical copies of $\zz$ and $\nn_0$ with the usual integers and nonnegative integers. An element is called \emph{rational} if it belongs to this copy of $\qq$, and \emph{irrational} otherwise. No negative rational belongs to $S$: if $-a/b \in S$ with $a, b \in \nn$ then $-1 = b(-a/b) + (a - 1) \in S$, contradicting additive reducedness. Hence $S \cap \qq \subseteq \qq_{\ge 0}$.
\smallskip

For positive real numbers $\alpha_1, \ldots, \alpha_t$, the semiring $\nn_0[\alpha_1, \ldots, \alpha_t]$ consists of the values $f(\alpha_1, \ldots, \alpha_t)$ with $f \in \nn_0[x_1, \ldots, x_t]$. A \emph{monogenic positive semiring} has the form $\nn_0[\alpha]$ for one positive real number $\alpha$. The generator is \emph{algebraic} if it is a root of a nonzero polynomial in $\qq[x]$, and \emph{transcendental} otherwise. A quadratic generator has degree two over $\qq$, meaning that the least degree of a nonzero polynomial in $\qq[x]$ vanishing at it is two. A rational monogenic semiring has a generator in $\qq_{> 0}$.
\smallskip

If $T \subseteq S^*$ is a multiplicative submonoid then the \emph{localization} of $S$ at $T$ is the subsemiring $T^{-1}S := \{s/t : s \in S,\ t \in T\}$ of the ambient field. This is the meaning of localization in the introductory discussion of~\cite{GPR24}.
\smallskip

An \emph{integral domain} is a commutative ring with $1 \ne 0$ in which a product of nonzero elements is nonzero. Factorization properties of an integral domain refer to its multiplicative monoid of nonzero elements; in particular, a \emph{unique factorization domain} is one for which that monoid is factorial. An \emph{ideal} of a commutative ring $R$ is an additive subgroup closed under multiplication by elements of $R$. An ideal is \emph{maximal} if it is proper and is contained in no larger proper ideal. A domain is \emph{quasilocal} if it has a unique maximal ideal $\mathfrak{m}$; its nonunits are then exactly the elements of $\mathfrak{m}$. The ideal $\mathfrak{m}^2$ consists of finite sums of products of two elements of $\mathfrak{m}$. For $p \in R$, the notation $pR$ denotes the principal ideal $\{pr : r \in R\}$.
\smallskip

Multiplicative factoriality forces a semidomain with a nonzero additive unit to be a unique factorization domain. The following observation will reduce the proof of Theorem~\ref{thm:upgrade} to the additively reduced case.
\smallskip

\begin{lemma}\label{lem:ring-or-reduced}
    Let $S$ be a semidomain such that $S^*$ is factorial. If $S$ is not additively reduced then $S$ is a unique factorization domain.
\end{lemma}
\smallskip

\begin{proof}
    Choose a nonzero additive unit $a$ of $S$, so that $a, -a \in S^*$. We have
    \[
        a^2 = (-a)^2.
    \]
    Comparing atom multiplicities in the unique multiplicative factorizations of these squares shows that $a$ and $-a$ are associates. Thus, $-a = ua$ for some $u \in S^\times$. Cancelling $a$ in the ambient field gives $u = -1$. Consequently, $-s = (-1)s \in S$ for every $s \in S$, so $S$ is a ring. As a subring of a field, it is an integral domain, and the factoriality of $S^*$ makes it a unique factorization domain.
\end{proof}
\smallskip

\medskip
\subsection{Groups and Factorization Coordinates}

For subsets $X$ and $Y$ of an additively written abelian group, their \emph{difference set} is
\[
    X - Y := \{x - y : x \in X,\ y \in Y\}.
\]
\smallskip

The \emph{quotient group} of a multiplicatively written cancellative monoid $M$ is obtained by adjoining formal fractions $a/b$, where $a, b \in M$ and $a/b = c/d$ exactly when $ad = bc$. For an additively written monoid, the corresponding construction is the \emph{Grothendieck group} $\gp(M)$ of formal differences: $a - b = c - d$ exactly when $a + d = b + c$. Addition is defined by $(a - b) + (c - d) := (a + c) - (b + d)$. Cancellation makes the natural map from the monoid to its group injective. In additive notation, this map is $\iota \colon M \to \gp(M)$, given by $\iota(a) := a - 0$.
\smallskip

The pair $(\gp(M), \iota)$ satisfies the following \emph{universal property}: for every abelian group $H$ and monoid homomorphism $f \colon M \to H$, there is a unique group homomorphism $\overline{f} \colon \gp(M) \to H$ such that $\overline{f} \circ \iota = f$. Explicitly, $\overline{f}(a - b) = f(a) - f(b)$ for all $a, b \in M$. This property characterizes the pair up to a unique isomorphism commuting with the canonical maps from $M$.
\smallskip

    For a semidomain $S$, we write $\gp_+(S)$ for the Grothendieck group of its additive monoid $(S, +)$. The multiplicative quotient group of $S$ identifies with $\{u/v : u, v \in S^*\}$ in the ambient field. The canonical isomorphism from $\gp_+(S)$ to the additive subgroup $S - S$ sends each formal difference to the corresponding difference in that field. We use this identification when writing $\gp_+(S) = S - S$. The latter is a ring, since products of differences are again differences of elements of $S$.
\smallskip

For an index set $I$, write $\nn_0^{(I)}$ and $\zz^{(I)}$ for the families of nonnegative integers and integers, respectively, indexed by $I$ and having finite support. Addition is coordinatewise. These are the \emph{free commutative monoid} and the \emph{free abelian group} on $I$. For a coordinate family $z = (z_i)_{i \in I}$, its \emph{support} is $\supp(z) := \{i \in I : z_i \ne 0\}$. The \emph{standard basis vector} $e_i$ has coordinate $1$ at $i$ and coordinate $0$ elsewhere. Every $z \in \zz^{(I)}$ has a unique decomposition $z = z^+ - z^-$ with disjoint supports, where $z_i^+ := \max\{z_i, 0\}$ and $z_i^- := \max\{-z_i, 0\}$. For $z \in \nn_0^{(I)}$, the sum of its coordinates is its length. Direct sums, denoted by $\oplus$, use these componentwise operations and require finite support when the indexing set is infinite.
\smallskip

    Let $M$ be a reduced atomic additive monoid, and denote its set of atoms by $\A_+(M)$. The \emph{factorization homomorphism} $\pi \colon \nn_0^{(\A_+(M))} \to M$ sends $\sum_{a \in \A_+(M)} z_a e_a$ to $\sum_{a \in \A_+(M)} z_a a$. The fiber over each element is its set of factorizations. It extends to a surjective group homomorphism $\pi \colon \zz^{(\A_+(M))} \to \gp(M)$. The \emph{relation group} is its kernel $\ker\pi := \{z \in \zz^{(\A_+(M))} : \pi(z) = 0\}$. A group is \emph{cyclic} if one element generates it; in additive notation, a cyclic subgroup generated by $\rho$ is $\zz\rho$. A nonzero integer coordinate vector is \emph{primitive} if the greatest common divisor of its coordinates is $1$.
\smallskip

\medskip
\subsection{Lattices and Cones}

An abelian group $L$ is \emph{torsion-free} if $ng = 0$ with $n \in \nn$ and $g \in L$ forces $g = 0$. For the rest of this subsection, let $L$ be a finitely generated torsion-free abelian group. Such a group has an integer basis and is therefore isomorphic to $\zz^d$ for some $d \in \nn_0$, called its \emph{rank}. The vector spaces $L \otimes_{\zz} \qq$ and $L \otimes_{\zz} \rr$ are obtained by allowing rational and real coefficients, respectively, on an integer basis of $L$. We identify $L$ with its natural image in these spaces. A \emph{lattice} in a finite-dimensional real vector space is the integer span of a real basis of that space. Thus, $L$ is a lattice in $L \otimes_{\zz} \rr$. This does not assert that a given embedding of $L$ into $\rr$ has discrete image; the coordinate space $L \otimes_{\zz} \rr$ can have dimension greater than one.
\smallskip

For $J \subseteq L$, write $\zz J$ for the subgroup of integer linear combinations of elements of $J$ and $\operatorname{span}_{\rr}(J)$ for their real linear span in $L \otimes_{\zz} \rr$. For a subgroup $L' \subseteq L$, the \emph{index} $[L:L']$ is the number of cosets of $L'$ in $L$. If this index is a finite integer $d$ then $dL \subseteq L'$, since the quotient group has order $d$. When $J$ is a finite linearly independent subset of $L$, the group $\zz J$ has finite index in $L \cap \operatorname{span}_{\rr}(J)$.
\smallskip

For a finite subset $B$ of a real vector space, its \emph{generated cone} is
\[
    \cone(B) := \left\{\sum_{b \in B} t_b b : t_b \in \rr_{\ge 0}\text{ for every }b \in B\right\}.
\]
Such a cone is closed in the finite-dimensional span of $B$. If $B \subseteq L$ then the \emph{lattice points} of this cone are the elements of $\cone(B) \cap L$. All limits involving these cones below are taken in the indicated real coordinate spaces.
\smallskip

\bigskip
\section{The Additive Upgrade}
\label{sec:upgrade}

We first reduce Theorem~\ref{thm:upgrade} to the additively reduced case. If $S$ satisfies its hypotheses and is not additively reduced then Lemma~\ref{lem:ring-or-reduced} makes $S$ a unique factorization domain. Thus, $(S, +)$ is a group and is factorial.

Throughout the rest of this section, let $S$ be an additively reduced semidomain. To prove Theorem~\ref{thm:upgrade}, suppose, by way of contradiction, that $(S, +)$ is a proper length-factorial monoid and $S^*$ is factorial.

    Let $\F := \nn_0^{(\A_+(S))}$ be the free commutative monoid on $\A_+(S)$.

    We identify $\gp_+(S)$ with $S - S$ as in Section~\ref{sec:background}.

    Let $e_a$ denote the basis vector corresponding to $a \in \A_+(S)$, and let $\pi \colon \zz^{(\A_+(S))} \to \gp_+(S)$ be the evaluation homomorphism determined by $\pi(e_a) = a$. The restriction of $\pi$ to $\F$ is the additive factorization homomorphism. All vectors in these direct sums have finite support; for $z \in \zz^{(\A_+(S))}$, we write $\supp(z) := \{a \in \A_+(S) : z_a \ne 0\}$. The group $\gp_+(S)$ is also a subring of the ambient field: the product of two differences of elements of $S$ is again a difference of elements of $S$. This observation is needed when powers of elements of $\gp_+(S)$ occur below. Since the ambient field has characteristic zero, $\gp_+(S)$ is torsion-free.

%%%%%%%%%%%%%%%%%%%%%%%%%%%%%%%%%%%%%%%%%%%%%%%%%%%%%%%%%%
\medskip
\subsection{The Relation Group and Preliminary Reductions}

    Let $K := \ker\pi$, and let $\lambda \colon \zz^{(\A_+(S))} \to \zz$ be the homomorphism that sums coordinates. If $z \in K \cap \ker\lambda$, write $z = z^{+} - z^{-}$ with $z^{+}, z^{-} \in \F$ and disjoint supports. Then $\pi(z^{+}) = \pi(z^{-})$ and $\lambda(z^{+}) = \lambda(z^{-})$. Length-factoriality gives $z^{+} = z^{-}$, so $z = 0$. Thus, $\lambda$ is injective on $K$. Since the additive monoid is not factorial, $K \ne 0$, and consequently
\begin{equation}\label{eq:cyclic-kernel}
    K = \zz\rho
\end{equation}
    for a nonzero vector $\rho$. The vector $\rho$ is \emph{primitive}, meaning that its coordinates have greatest common divisor $1$. Indeed, if $\rho = d\eta$ with $d > 1$ and $\eta \in \zz^{(\A_+(S))}$ then torsion-freeness of $\gp_+(S)$ gives $\eta \in K$, contradicting $K = \zz\rho$. Moreover, $\lambda(\rho) \ne 0$, and $\rho$ has both positive and negative coordinates because $\pi(\rho) = 0$ and a nonempty sum of nonzero elements of $S$ cannot be zero.

\begin{lemma}[Basic reductions]\label{lem:basic}
    The following statements hold.
    \begin{enumerate}[label=(\arabic*),ref=\arabic*]
            \item\label{item:basic-units} $1 \in \A_+(S)$ and $S^\times = \{1\}$.
        \smallskip

        \item\label{item:basic-divisors} Every multiplicative divisor of an additive atom is an additive atom.
        \smallskip

        \item\label{item:basic-rationals} $S \cap \qq = \nn_0$.
        \smallskip

        \item\label{item:basic-integers} For each $n \in \nn$, the only additive factorization of $n$ is $ne_1$.
    \end{enumerate}
\end{lemma}

\begin{proof}
        (\ref{item:basic-units}) Choose an additive atom $a$, which exists by atomicity since $1 \ne 0$. If $1 = x + y$ for nonzero $x, y \in S$ then $a = ax + ay$ contradicts the additive irreducibility of $a$. Thus, $1 \in \A_+(S)$.

        If $u \in S^\times$, multiplication by $u$ permutes $\A_+(S)$ and induces a permutation automorphism $T_u$ of $\zz^{(\A_+(S))}$. It preserves $K$ and $\lambda$, so $T_u(\rho) = \varepsilon\rho$ for some $\varepsilon \in \{1, -1\}$. Since $\lambda(\rho) \ne 0$, comparison of lengths gives $\varepsilon = 1$. Hence multiplication by $u$ permutes the finite nonempty set $\supp(\rho)$. For some $c \in \supp(\rho)$ and $j \in \nn$, we have $u^jc = c$. Cancelling the nonzero element $c$ gives $u^j = 1$. If $u \ne 1$ then the geometric sum identity in the ambient field gives $1 + u + \cdots + u^{j - 1} = 0$, contradicting additive reducedness. Thus, $u = 1$.
    \smallskip

        (\ref{item:basic-divisors}) If $\alpha = xy \in \A_+(S)$ with $x, y \in S^*$ and $x = r + s$ for nonzero $r, s \in S$ then $\alpha = ry + sy$ is a nontrivial additive decomposition. Hence $x \in \A_+(S)$, and similarly $y \in \A_+(S)$.
    \smallskip

        (\ref{item:basic-rationals}) By the observations in Section~\ref{sec:background}, $S \cap \qq \subseteq \qq_{\ge 0}$. Take $q = a/b \in S \cap \qq_{> 0}$ with $a, b \in \nn$. Choose $X \in \F$ with $\pi(X) = q$. Then $bX - ae_1 \in \zz\rho$. Put $E := \{1\} \cup \supp(\rho)$, and let $\zz E$ be the subgroup of $\gp_+(S)$ generated by $E$. The preceding relation implies $\supp(X) \subseteq E$, so $q \in \zz E$. For each $c \in E$, choose a factorization $Y_c$ of $qc$. The relation $bY_c - ae_c \in \zz\rho$ similarly gives $\supp(Y_c) \subseteq E$. Thus, $q(\zz E) \subseteq \zz E$. The group $\zz E$ is nonzero, finitely generated, and torsion-free. Multiplication by $q$ is therefore represented by an integer matrix on a basis of $\zz E$. On $\zz E \otimes_{\zz} \qq$, this map acts as $qI$. The integer matrix is therefore $qI$, so its diagonal entries give $q \in \zz$. Conversely, $\nn_0 \subseteq S$ because $S$ contains $1$. This proves~(\ref{item:basic-rationals}).
    \smallskip

        (\ref{item:basic-integers}) Suppose that $n$ has a factorization different from $ne_1$. Cancel the common copies of $1$. We obtain $X' \in \F \setminus \{0\}$ and $m \in \nn$ with $1 \notin \supp(X')$ and $\pi(X') = m$. After orienting $\rho$, we have $X' - me_1 = h\rho$ for some $h \in \nn$. Thus, for distinct $a_1, \ldots, a_s \in \A_+(S) \setminus \{1\}$ and $d, r_1, \ldots, r_s \in \nn$, we have
    \begin{equation}\label{eq:integer-relation}
        d = \sum_{i = 1}^s r_i a_i,\qquad
        \rho = \sum_{i = 1}^s r_i e_{a_i} - de_1.
    \end{equation}
    The first equality is obtained by applying $\pi$ to $\rho$. The coefficients are integers because they are coordinates of $\rho$.

    If $s = 1$ then $a_1 = d/r_1 \in S \cap \qq = \nn_0$, which is impossible: the only positive integer that can be an additive atom is $1$. If $s \ge 2$, choose a factorization $Z_i$ of $a_1a_i$ for each $i$ and put $Z := \sum_i r_iZ_i$. Multiplying the first equality in~\eqref{eq:integer-relation} by $a_1$ gives $\pi(Z) = da_1$. Hence $Z - de_{a_1} = k\rho$ for some $k \in \zz$. At coordinate $1$, nonnegativity gives $k \le 0$; at coordinate $a_2$, it gives $k \ge 0$. Thus, $k = 0$ and $Z = de_{a_1}$. Every $Z_i$ is therefore supported on $\{a_1\}$. In particular, $a_1a_2 = ta_1$ for some $t \in \nn$. Cancellation gives $a_2 = t$, again a contradiction.
\end{proof}

\begin{lemma}[An irrational common atom]\label{lem:common}
    There exists an irrational element of $\A_+(S) \cap \A_\times(S)$.
\end{lemma}
\begin{proof}
    Since $(S, +)$ is not factorial, $S \ne \nn_0$. Choose $s \in S \setminus \nn_0$. By Lemma~\ref{lem:basic}(\ref{item:basic-rationals}), some atom in an additive factorization of $s$ is irrational; call it $\alpha$. By Lemma~\ref{lem:basic}(\ref{item:basic-units}), $\alpha$ is a multiplicative nonunit. Factor it multiplicatively. Every factor is also an additive atom by Lemma~\ref{lem:basic}(\ref{item:basic-divisors}), and at least one factor is irrational because their product is irrational.
\end{proof}

\medskip
\subsection{A Cone Argument for Divisibility}

The next lemma uses a finitely generated cone to produce a uniform multiplier. We state it for an additively reduced semidomain whose additive relation group is finitely generated; multiplicative factoriality is not needed. Applied to $S$, it combines with Lemma~\ref{lem:ray} to yield cancellation of multiplicative divisibility across a sum. This cancellation will establish the sign information needed for the integer-divisor obstruction in Lemma~\ref{lem:integer-obstruction}.

\begin{lemma}[A uniform multiplier for powers along a ray]\label{lem:cone}
        Let $T$ be an additively reduced semidomain whose additive monoid is atomic and has finitely generated relation group. There exists $D \in \nn$, depending only on $(T, +)$, with the following property. If $c \in \gp_+(T)$ and $y \in T$ satisfy $y + nc \in T$ for every $n \in \nn_0$ then $Dc^j \in T$ for every $j \in \nn$.
\end{lemma}
\begin{proof}
        We identify $\gp_+(T)$ with the ring $T - T$ in the ambient field. Let $\pi_T \colon \zz^{(\A_+(T))} \to \gp_+(T)$ be the evaluation homomorphism, and let $K_T$ be its kernel. Choose a finite generating set of $K_T$, and let $B$ be the union of the supports of its elements. If $K_T = 0$, choose the empty generating set, so $B = \varnothing$. Put $C := \A_+(T) \setminus B$. Let $M_B$ be the additive monoid generated by $B$, and put $L := \gp(M_B)$. The group $L$ is finitely generated and torsion-free, so it is a lattice in $L \otimes_{\zz} \rr$. Since every relation is supported on $B$, evaluation gives the identification
    \[
        (T, +) \cong M_B \oplus \nn_0^{(C)}.
    \]
    Indeed, equal expressions have the same $C$-coordinates and equal evaluations of their $B$-parts. Passing to Grothendieck groups gives
    \[
            \gp_+(T) \cong L \oplus \zz^{(C)}.
    \]
    Let $P := \cone(B) \subseteq L \otimes_{\zz} \rr$. Here the cone consists of all nonnegative real linear combinations of the vectors in $B$; since $B$ is finite, this cone is closed. There is an integer $D \ge 1$ such that
    \begin{equation}\label{eq:uniform-multiplier}
        D(P \cap L) \subseteq M_B.
    \end{equation}

    If $B = \varnothing$ then $P = L = M_B = \{0\}$, so we may take $D = 1$. Otherwise, take a nonzero element of $P$ and choose a nonnegative real expression for it with minimal support $J \subseteq B$. If $J$ were linearly dependent, one could subtract a suitable multiple of a dependence from the coefficients, keeping them nonnegative and making at least one of them zero. This would contradict minimality. Hence $J$ is linearly independent. The subgroup $\zz J$ has finite index
    \[
        d_J := [L \cap \operatorname{span}_{\rr}(J):\zz J]
    \]
    in the lattice in its real span. For a lattice point expressed on $J$, multiplication by $d_J$ makes its unique $J$-coordinates integral and nonnegative. Taking $D$ to be the least common multiple of $d_J$ over the nonempty linearly independent subsets $J \subseteq B$ proves~\eqref{eq:uniform-multiplier}; the zero point requires no separate argument. Thus, $D$ is independent of $c$, $y$, and the exponent $j$.

        Fix $j \in \nn$. We use the direct-sum identifications above to regard elements of $\gp_+(T)$ as coordinate vectors. Since $\gp_+(T)$ is a ring, the coefficients in
    \[
        (y + nc)^j = \sum_{i = 0}^j \binom{j}{i} n^i y^{j - i}c^i
    \]
        belong to $\gp_+(T)$. Only finitely many $C$-coordinates occur among these coefficients; choose a finite set $E \subseteq C$ containing all of them. Every $(y + nc)^j$ lies in the closed cone $P \times \rr_{\ge 0}^{E}$ in $(L \otimes_{\zz} \rr) \oplus \rr^E$. Divide the displayed polynomial identity by $n^j$ for $n \ge 1$ and let $n$ tend to infinity. Its limit is $c^j$, so $c^j$ belongs to that cone. Its $L$-component is a lattice point of $P$, and its $E$-coordinates are nonnegative integers. Equation~\eqref{eq:uniform-multiplier} gives $Dc^j \in M_B \oplus \nn_0^{(C)} = T$. This limit is taken in the displayed finite-dimensional real coordinate space and requires no ordering of the ambient field.
\end{proof}

We return to the standing semidomain $S$. Its relation group $K = \zz\rho$ is finitely generated, so Lemma~\ref{lem:cone} applies. The next lemma constructs the ray to which we will apply it.

\begin{lemma}[A ray from a divisibility equation]\label{lem:ray}
    Let $r \in \A_\times(S)$, and suppose that $m + ry = rz$ for $m, y, z \in S$ with $m \ne 0$. Then $y + Nm/r \in S$ for every $N \in \nn$.
\end{lemma}
\begin{proof}
    Fix $N \in \nn$, choose an integer $L \ge N$, and put $u_N := Nm + ry$. Using $m + ry = rz$ in the first two terms of the binomial expansion gives
    \begin{align*}
        u_N^L = r\bigg(&N^L m^{L - 1}z
        + N^{L - 1}(L - N)m^{L - 1}y\\
        & + \sum_{i = 2}^L \binom{L}{i} N^{L - i}m^{L - i}r^{i - 1}y^i\bigg).
    \end{align*}
    All coefficients inside the parentheses are nonnegative integers, and $u_N \ne 0$ by additive reducedness. Thus, $r \mid_{S^*} u_N^L$. Since $S^*$ is factorial, $r$ is prime, and hence $r \mid_{S^*} u_N$. Dividing by $r$ in the ambient field gives $y + Nm/r \in S$.
\end{proof}

The uniform multiplier now permits cancellation of divisibility by a multiplicative atom. Factoring an arbitrary nonzero divisor extends the conclusion to all of $S^*$.

\begin{lemma}[Cancellation of divisibility across a sum]\label{lem:cancel}
    If $d \in S^*$ and $m + dy = dz$ with $m, y, z \in S$ and $m \ne 0$ then $m/d \in S$.
\end{lemma}
\begin{proof}
        First suppose that $d = r \in \A_\times(S)$. Put $c := m/r = z - y \in \gp_+(S)$. Lemma~\ref{lem:ray} gives $y + Nc \in S$ for every $N \in \nn_0$. By Lemma~\ref{lem:cone}, $Dc^j \in S^*$ for all $j \in \nn$. Consequently,
    \[
        r^j(Dc^j) = Dm^j,
    \]
    so $r^j \mid_{S^*} Dm^j$ for every $j$. If $r \nmid_{S^*} m$ then unique factorization would imply $r^j \mid_{S^*} D$ for every $j$. The exponent of $r$ in the factorization of the fixed nonzero element $D$ would therefore be at least every positive integer, which is impossible. Hence $r \mid_{S^*} m$. The independence of $D$ from $j$ is essential here.

    For general $d \in S^*$, we argue by induction on its factorization length. If $d$ is a unit then $m/d \in S^*$. Otherwise, write $d = re$ with $r \in \A_\times(S)$ and $e \in S^*$. Applying the atomic case to $m + r(ey) = r(ez)$ gives $m/r \in S^*$. Dividing the original equality by $r$ gives $m/r + ey = ez$. The induction hypothesis applied to $e$ now gives $m/d \in S$.
\end{proof}

The following consequences exclude the divisibility equations needed in the rest of the argument.

\begin{corollary}\label{cor:forbidden}
    The following statements hold.
    \begin{enumerate}[label=(\arabic*),ref=\arabic*]
        \item\label{item:forbidden-unit} For every multiplicative nonunit $d \in S^*$, there are no $y, z \in S$ with $1 + dy = dz$.
        \smallskip

            \item\label{item:forbidden-atom} For every $d \in S^*$ and $\beta \in \A_+(S)$, there are no $m, y \in S^*$ with $m + dy = d\beta$.
    \end{enumerate}
\end{corollary}
\begin{proof}
    (\ref{item:forbidden-unit}) Lemma~\ref{lem:cancel} would give $1/d \in S$, contradicting that $d$ is a nonunit.
    \smallskip

    (\ref{item:forbidden-atom}) Lemma~\ref{lem:cancel} would give $m/d \in S^*$ and hence $\beta = m/d + y$, contradicting the additive irreducibility of $\beta$.
\end{proof}
Assertion~(\ref{item:forbidden-unit}) will exclude nonunit common divisors in the next lemma. Assertion~(\ref{item:forbidden-atom}) will constrain the additive factorizations used in the proof of Theorem~\ref{thm:upgrade}.

\medskip
\subsection{The Polynomial Identity and the Signs of the Relation}

We now combine a polynomial identity with cancellation of divisibility to determine the signs of the coordinates of $\rho$ corresponding to $1$ and to irrational common atoms.

\begin{lemma}\label{lem:polynomial}
    For each irrational common atom $p \in \A_+(S) \cap \A_\times(S)$, we have $p^2 - p + 1 \in S$ and $p^2 - p \notin S$.
\end{lemma}
\begin{proof}
    Consider the identity
    \[
        (1 + p + p^2)(1 + p^3) = (1 + p)(1 + p^2 + p^4).
    \]
    Let $u$ be a greatest common divisor of $1 + p + p^2$ and $1 + p$ in the factorial monoid $S^*$. If a multiplicative atom $q$ divides $u$ then $q \mid_{S^*} 1 + p$ and $q \mid_{S^*} 1 + p + p^2$. Writing $1 + p = qw$ and $1 + p + p^2 = qv$, with $v, w \in S^*$, the identity $1 + p + p^2 = 1 + p(1 + p)$ becomes $1 + q(pw) = qv$. This contradicts Corollary~\ref{cor:forbidden}(\ref{item:forbidden-unit}). Thus, $u$ is a unit, and Euclid's lemma applied to the displayed product gives $1 + p \mid_{S^*} 1 + p^3$. Therefore
    \[
        p^2 - p + 1 = \frac{1 + p^3}{1 + p} \in S.
    \]

    If $s := p^2 - p$ belonged to $S$ then $s \ne 0$ because $p$ is irrational. The equation $s + p = p^2$ would contradict Corollary~\ref{cor:forbidden}(\ref{item:forbidden-atom}), with $d = \beta = p$ and $y = 1$.
\end{proof}

We now use these membership and nonmembership assertions to choose an orientation of $\rho$ that works for every irrational common atom.

\begin{lemma}[Orientation of the relation]\label{lem:sign}
    We may orient $\rho$ so that $\rho_1 < 0$. With this orientation, $\rho_p > 0$ for every irrational common atom $p$.
\end{lemma}
\begin{proof}
    Choose an irrational common atom $p$ by Lemma~\ref{lem:common}. Let $X$ and $P_2$ be additive factorizations of $p^2 - p + 1$ and $p^2$, respectively. By Lemma~\ref{lem:polynomial}, $X_1 = 0$ and $(P_2)_p = 0$. Indeed, removing a copy of $1$ from $X$, or a copy of $p$ from $P_2$, would give an element $p^2 - p$ of $S$. The identity $(p^2 - p + 1) + p = 1 + p^2$ yields
    \[
        X + e_p - e_1 - P_2 \in K.
    \]
    This relation has a negative $1$-coordinate and a positive $p$-coordinate. Orienting $\rho$ by $\rho_1 < 0$ therefore gives $\rho_p > 0$. The same argument works for every irrational common atom with this fixed orientation.
\end{proof}

\medskip
\subsection{The Integer-Divisor Obstruction}

The following lemma provides the integer-divisor obstruction needed to complete the proof of Theorem~\ref{thm:upgrade}. Its proof uses the signs established in Lemma~\ref{lem:sign}.

\begin{lemma}[Integer-divisor obstruction]\label{lem:integer-obstruction}
        Suppose that $a \in \A_+(S)$ is irrational and $\rho_1 < 0 < \rho_a$. Then $a$ does not divide a positive integer in $S^*$.
\end{lemma}
\begin{proof}
    Suppose that $ab = n$ for some $b \in S^*$ and $n \in \nn$. If the positive support of $\rho$ is $\{a\}$ then its evaluated relation has the form
    \[
        ca = \sum_{\beta : \rho_\beta < 0}v_\beta\beta
    \]
    for some $c, v_\beta \in \nn$ with $v_1 > 0$. Multiplication by $b$ gives an expression for the integer $cn$ as a sum of the elements $v_\beta b\beta$. By Lemma~\ref{lem:basic}(\ref{item:basic-integers}), every additive factorization of every $b\beta$ appearing here uses only $1$. In particular, $b \in \nn$. Then $a = n/b$ is rational, a contradiction.

        Assume that the positive support has at least two elements. The negative support also has at least two elements. Indeed, it contains $1$; if it were $\{1\}$ then the relation would give a positive integer an additive factorization different from the one consisting only of copies of $1$, contrary to Lemma~\ref{lem:basic}(\ref{item:basic-integers}). It follows that for every $\beta \in \A_+(S)$ and $m \in \nn$, the unique additive factorization of $m\beta$ is $me_\beta$. To see this, write another factorization as $me_\beta + t\rho$. A positive coordinate of $\rho$ away from $\beta$ forces $t \ge 0$, and a negative coordinate away from $\beta$ forces $t \le 0$.

        Choose additive factorizations of $a\beta$ and $b\beta$ for each $\beta \in \A_+(S)$. They define homomorphisms $U, V \colon \zz^{(\A_+(S))} \to \zz^{(\A_+(S))}$ with nonnegative integer columns, lifting multiplication by $a$ and $b$, respectively. The uniqueness just proved gives
    \begin{equation}\label{eq:positive-inverses}
        UV = VU = nI.
    \end{equation}
        Since $\zz^{(\A_+(S))}$ is torsion-free, these identities imply that $U$ and $V$ are injective. We claim that
    \begin{equation}\label{eq:monomial-map}
        Ue_\beta = c_\beta e_{\sigma(\beta)}
    \end{equation}
        for every $\beta \in \A_+(S)$, where $c_\beta \in \nn$ and $\sigma$ is a permutation of $\A_+(S)$.

        To justify this assertion, choose $\gamma \in \supp(Ue_\beta)$. The identity $VUe_\beta = ne_\beta$ and nonnegativity imply $Ve_\gamma = d_\gamma e_\beta$ for some $d_\gamma \in \nn$. The coefficient is positive because $UV = nI$ ensures that $Ve_\gamma \ne 0$. Applying $U$ and using $UVe_\gamma = ne_\gamma$ shows that $d_\gamma Ue_\beta = ne_\gamma$, so $Ue_\beta$ has just one nonzero coordinate. The map $\sigma$ is injective: if two columns of $U$ had the same support then a nonzero integer linear combination of those columns would vanish, contrary to the injectivity of $U$. It is surjective because $UVe_\gamma = ne_\gamma$ for every $\gamma$, so every basis coordinate occurs in the image of some column of $U$. This argument works even when $\A_+(S)$ is infinite.

    Since $\pi U = a\pi$, the map $U$ preserves $K$. Thus, there is an integer $k \ne 0$ such that $U\rho = k\rho$; here $k \ne 0$ follows from the injectivity of $U$. By~\eqref{eq:monomial-map}, distinct coordinates are sent to distinct coordinates with positive coefficients. Consequently, $\supp(U\rho) = \sigma(\supp(\rho)) = \supp(\rho)$, so $\sigma$ permutes the finite set $\supp(\rho)$. Choose $\beta$ in that set and let $h \in \nn$ be the length of its $\sigma$-cycle. Comparing its coefficient in $U^h\rho = k^h\rho$ gives
    \[
        \prod_{i = 0}^{h - 1}c_{\sigma^i(\beta)} = k^h.
    \]
    On the other hand, evaluating $U^he_\beta$ gives $a^h\beta = k^h\beta$. Cancelling the nonzero element $\beta$ gives $a^h = k^h$. Squaring and factoring in the ambient field yield
    \[
        0 = a^{2h} - |k|^{2h} = (a - |k|)\sum_{i = 0}^{2h - 1}a^i|k|^{2h - 1 - i}.
    \]
    Here $|k|$ denotes the absolute value of the integer $k$. Every summand in the second factor is a nonzero element of $S$, so their sum is nonzero by additive reducedness. Hence $a = |k| \in \nn$, contradicting the irrationality of $a$.
\end{proof}

%%%%%%%%%%%%%%%%%%%%%%%%%%%%%%%%%%%%%%%%%%%%%%
\medskip
\subsection{Completion of the Additive Upgrade}

We conclude this section by proving Theorem~\ref{thm:upgrade}.

\begin{proof}[Proof of Theorem~\ref{thm:upgrade}]
    By the reduction at the start of Section~\ref{sec:upgrade}, it suffices to work under its standing assumptions. Choose an irrational common atom $a$. Orient $\rho$ as in Lemma~\ref{lem:sign}. Put $x := a^2 - a + 1 \in S$, so that
    \begin{equation}\label{eq:quadratic-identity}
        x + a = 1 + a^2.
    \end{equation}
        When splitting an additive factorization according to divisibility by $a$, each atom $\alpha \in aS$ can be written as $a\beta$ with $\beta \in \A_+(S)$, by Lemma~\ref{lem:basic}(\ref{item:basic-divisors}). Thus, the evaluated factorization can be written as $\pi(Z) + a\pi(Y)$, where $Z, Y \in \F$ and $Z$ is supported on atoms outside $aS$. This decomposition is an equality in $S$. In particular, this splitting uses evaluation in $S$ and does not presuppose a multiplication map on $\F$.

    \begin{claim}\label{claim:unique-products}
            For every $\beta \in \A_+(S)$, the element $a\beta$ has a unique additive factorization.
    \end{claim}
    \begin{proof}[\mathsc{Proof of Claim~\ref{claim:unique-products}}]
        \let\qed\relax
        If $a\beta$ is an additive atom then its only factorization is $e_{a\beta}$. Otherwise, consider a factorization of $a\beta$. It cannot contain both an atom in $aS$ and an atom outside $aS$, since such a splitting would give $m + ay = a\beta$ with $m, y \in S^*$, contrary to Corollary~\ref{cor:forbidden}(\ref{item:forbidden-atom}). Nor can all of its atoms belong to $aS$: dividing their sum by $a$ and using that $\beta$ is an additive atom would force the factorization to have only one atom. Hence every factorization is supported outside $aS$. Every two such factorizations have zero $a$-coordinate; their difference lies in $\zz\rho$, and $\rho_a > 0$, so they are equal. Therefore Claim~\ref{claim:unique-products} is established.
    \end{proof}

        Define $B \colon \zz^{(\A_+(S))} \to \zz^{(\A_+(S))}$ by sending $e_\beta$ to the unique factorization of $a\beta$ supplied by Claim~\ref{claim:unique-products}. We have $B(e_1) = e_a$. For $\beta \ne 1$, the $a$-coordinate of $B(e_\beta)$ is zero: this was proved above when $a\beta$ is not an additive atom; when it is an additive atom, $a\beta = a$ would force $\beta = 1$.

    \begin{claim}\label{claim:positive-support}
            There is an atom $\gamma \ne 1$ such that the factorization $B(e_\gamma)$ is supported outside $aS$ and contains both $1$ and an atom different from $1$. Moreover, every $\beta \in \A_+(S)$ with $\rho_\beta > 0$ belongs to $\{a, \gamma\}$.
    \end{claim}
    \begin{proof}[\mathsc{Proof of Claim~\ref{claim:positive-support}}]
        \let\qed\relax
        Since $\pi B = a\pi$, the map $B$ preserves $K$. Write $B(\rho) = t\rho$ for $t \in \zz$. Comparing $a$-coordinates gives $t\rho_a = \rho_1$, so $t < 0$. Write $t = -m$ with $m \in \nn$. Comparison at coordinate $1$ gives
        \[
                \sum_{\beta \in \A_+(S)}\rho_\beta\bigl(B(e_\beta)\bigr)_1
            = -m\rho_1 > 0.
        \]
            Consequently, there is $\gamma \in \A_+(S)$ with $\rho_\gamma > 0$ and $\bigl(B(e_\gamma)\bigr)_1 > 0$. Every coefficient of $B(e_\beta)$ is nonnegative, so a positive sum forces at least one positive contribution. Since $\rho_\gamma > 0 > \rho_1$, we have $\gamma \ne 1$. The factorization $B(e_\gamma)$ of $a\gamma$ contains $1$, and it is supported entirely outside $aS$: a mixed factorization was ruled out above, and $1 \notin aS$.

        The atom $a$ is irrational and satisfies $\rho_1 < 0 < \rho_a$, so Lemma~\ref{lem:integer-obstruction} gives $a\gamma \notin \nn$. Thus, this factorization also contains an atom different from $1$. Remove one copy of $1$ and write $a\gamma = 1 + u$, with $u \in S^*$. Fix for $u$ the remaining factorization, which is still supported outside $aS$ and contains an atom different from $1$. Equation~\eqref{eq:quadratic-identity} gives
        \[
            u + x + a = a(\gamma + a).
        \]
        Factor $x$ additively, retain the fixed factorization of $u$, and retain the explicit summand $a$. Split the resulting factorization as $\pi(Z) + a\pi(Y)$ according to divisibility by $a$. Then $Z \ne 0$ because of the summands from $u$, and $Y_1 \ge 1$ because of the explicit $a$. Lemma~\ref{lem:cancel} gives $\pi(Z)/a \in S$. Dividing by $a$ gives $\pi(Z)/a + \pi(Y) = \gamma + a$. Adding an additive factorization of $\pi(Z)/a$ to $Y$ therefore produces a factorization $E$ of $\gamma + a$ with $E_1 \ge Y_1 > 0$.

        Write $E - e_\gamma - e_a = t'\rho$. Since $\gamma \ne 1$, its $1$-coordinate is positive, whereas $\rho_1 < 0$; hence $t' < 0$. At each coordinate outside $\{a, \gamma\}$ the left side is nonnegative, so $\rho_\beta \le 0$ there. Therefore
        \begin{equation}\label{eq:positive-support}
                \{\beta \in \A_+(S) : \rho_\beta > 0\} \subseteq \{a, \gamma\}.
        \end{equation}
        Therefore Claim~\ref{claim:positive-support} is established.
    \end{proof}

    Fix $\gamma$ as in Claim~\ref{claim:positive-support}.

    \begin{claim}\label{claim:common-gamma}
        The element $\gamma$ is an irrational common atom distinct from $a$, and it divides an atom $\delta \ne 1$ occurring in $B(e_\gamma)$.
    \end{claim}
    \begin{proof}[\mathsc{Proof of Claim~\ref{claim:common-gamma}}]
        \let\qed\relax
        Choose an atom $\delta \ne 1$ in the chosen factorization of $a\gamma$. It is irrational by Lemma~\ref{lem:basic}(\ref{item:basic-rationals}), and $\delta \notin aS$. Factor $\delta$ multiplicatively. As in Lemma~\ref{lem:common}, it has an irrational common atom divisor $p$. Since $\delta \notin aS$, we have $p \ne a$. Lemma~\ref{lem:sign} gives $\rho_p > 0$, so \eqref{eq:positive-support} forces $p = \gamma$. Thus, $\gamma$ is an irrational common atom distinct from $a$, and $\gamma \mid_{S^*} \delta$. Therefore Claim~\ref{claim:common-gamma} is established.
    \end{proof}

    Fix $\delta$ as in Claim~\ref{claim:common-gamma}. Split the chosen additive factorization of $a\gamma$ according to divisibility by $\gamma$. The occurrence of $1$ supplies a nonzero outside part, since $\gamma$ is a nonunit. The occurrence of $\delta$ supplies a nonzero divisible part. We obtain
    \[
        m' + \gamma y' = \gamma a
    \]
    for some $m', y' \in S^*$, contrary to Corollary~\ref{cor:forbidden}(\ref{item:forbidden-atom}). This contradiction proves that $(S, +)$ is factorial.
\end{proof}

%%%%%%%%%%%%%%%%%%%%%%%%%%%%%%%%%%%%%%%%%%%%%%%%%%%%%%%%%%%%%%%%%%
%%%%%%%%%%%%%%%%%%%%%%%%%%%%%%%%%%%%%%%%%%%%%%%%%%%%%%%%%%%%%%%%%%
\bigskip
\section{From the Additive Upgrade to the Classification}\label{sec:completion}

We complete the proof of Theorem~\ref{thm:bilf} by treating the additively reduced bi-factorial case and then applying Lemma~\ref{lem:ring-or-reduced} to handle the remaining semidomains. We include the reduced-case proof for completeness.

\begin{lemma}[The additively reduced bi-factorial case]\label{lem:biuf}
    Let $S$ be an additively reduced semidomain. If both $(S, +)$ and $S^*$ are factorial then $S$ is isomorphic to $\nn_0$.
\end{lemma}

\begin{proof}
        The additive monoid is free on $\A_+(S)$. The argument proving $1 \in \A_+(S)$ at the start of Lemma~\ref{lem:basic} uses only additive atomicity and additive reducedness and therefore applies here. For each $b \in \A_+(S)$, let $\lambda_b \colon \gp_+(S) \to \zz$ denote the corresponding additive coordinate. Then $g \in \gp_+(S)$ belongs to $S$ exactly when all its coordinates are nonnegative.

        Fix $x \in S$ and $r \in \nn$. Additive reducedness ensures that $x + n \ne 0$ for every $n \in \nn$. No multiplicative nonunit $d$ divides every $x + n$ with $n \ge r$. Indeed, if $x + n = dy_n$ then $d(y_{n + 1} - y_n) = 1$. Thus, for $c := y_{r + 1} - y_r \in \gp_+(S)$,
    \[
        dc = 1,\qquad y_n = y_r + (n - r)c.
    \]
    If some coordinate of $c$ were negative, the corresponding coordinate of $y_n$ would eventually be negative. Therefore $c \in S$, which would make $d$ a unit. Factor $x + r$ multiplicatively. Each atom in this finite factorization fails to divide some later $x + n$; choosing $N \ge r$ past all these failures gives a block
    \begin{equation}\label{eq:finite-block}
        x + r, x + r + 1, \ldots, x + N
    \end{equation}
    with no common nonunit divisor. If $x + r$ is a unit, take $N = r$. Otherwise, every common nonunit divisor would have an atom dividing $x + r$, which is exactly what the construction excludes.

    We also use denominator clearing in the factorial monoid $S^*$. Let $t \in \nn$. Suppose that $q$ belongs to the quotient group of $S^*$, that $s_1, \ldots, s_t \in S^*$ have no common nonunit divisor, and that $qs_i \in S^*$ for every $i$. Write $q = u/v$ with $u, v \in S^*$ and cancel their common multiplicative atoms. The relations $us_i = v(qs_i)$ imply $v \mid_{S^*} s_i$ for every $i$. Thus, $v$ is a unit and $q \in S^*$.

        Suppose that $a \in \A_+(S) \setminus \{1\}$. Put $k := \lambda_a(a^2) + 1$, and choose $N \ge k$ so that $a + k, \ldots, a + N$ have no common nonunit divisor. Define $H := a^2 - ka + kN \in \gp_+(S)$. We have
    \begin{equation}\label{eq:clearing-polynomial}
        (a + n)H = a^3 + (n - k)a^2 + k(N - n)a + nkN \in S^*
    \end{equation}
    for every integer $n$ with $k \le n \le N$. All coefficients are nonnegative integers, and the nonzero summand $nkN$ ensures that these products are nonzero by additive reducedness. In particular, $H = ((a + k)H)/(a + k)$ belongs to the multiplicative quotient group. Denominator clearing gives $H \in S$. But
    \[
        \lambda_a(H) = \lambda_a(a^2) - k + kN\lambda_a(1) = -1,
    \]
    a contradiction. Thus, $1$ is the only additive atom, and atomicity gives $S = \nn_0$.
\end{proof}

With the tools we have developed, we are in a good position to classify all bi-length-factorial semidomains and thereby prove the Bi-LF Positive Conjecture.

\begin{proof}[Proof of Theorem~\ref{thm:bilf}]
    Suppose that $(S, +)$ and $S^*$ are length-factorial. By Bu, Vulakh, and Zhao~\cite[Corollary 5.4]{BVZ23}, the multiplicative monoid of a semidomain is length-factorial if and only if it is factorial. Thus, $S^*$ is factorial. Theorem~\ref{thm:upgrade} implies that $(S, +)$ is factorial. If $S$ is additively reduced then Lemma~\ref{lem:biuf} gives $S \cong \nn_0$. Otherwise, Lemma~\ref{lem:ring-or-reduced} shows that $S$ is a unique factorization domain.

    Conversely, $\nn_0$ has factorial additive and multiplicative monoids. If $S$ is a unique factorization domain then $S^*$ is factorial and $(S, +)$ is a group, hence factorial. Therefore both alternatives in the statement give bi-factorial, and consequently bi-length-factorial, semidomains.
\end{proof}

%%%%%%%%%%%%%%%%%%%%%%%%%%%%%%%%%%%%%%%%%%%%%%%%%%
%%%%%%%%%%%%%%%%%%%%%%%%%%%%%%%%%%%%%%%%%%%%%%%%%%
\bigskip
\section{Quasi-Factoriality and Related Questions}
\label{sec:quasi-factoriality}

We conclude by considering uniqueness of factorizations of a prescribed length. For a fixed integer $n \ge 2$, quasi-$n$-factoriality implies quasi-$2$-factoriality. Indeed, if two distinct factorizations of length two existed then adjoining $n - 2$ copies of the same atom to each would give two distinct factorizations of length $n$. The distinction is preserved because adjoining the same multiset of atoms to two different multisets cannot make them equal.

%%%%%%%%%%%%%%%%%%%%%%%%%%%%%
\medskip
\subsection{Integral Domains}

The following question is equivalent to the open problem at the end of~\cite{CC26}, in view of the preceding observation and the equivalence between length-factoriality and factoriality for integral domains~\cite{CS11}.

\begin{question}\label{question:quasi-two-domain}
    Must every atomic quasi-$2$-factorial integral domain be length-factorial, or equivalently a unique factorization domain?
\end{question}

Every quasi-$2$-factorial monoid is square-factorial, since the two factorizations of length two in an equality $u^2 = vw$ between products of atoms must agree up to order and associates.

The additive structure of an integral domain provides a way to construct relations of length two. Let $R$ be an integral domain, and suppose that an atom $p \in R$ is not prime. Choose $x, y \in R \setminus pR$ and $z \in R$ with $xy = pz$. Then
\begin{equation}\label{eq:quasi-fundamental}
    (p + x)(p + y) = p^2 + p(x + y) + xy = p(p + x + y + z).
\end{equation}
If the three elements $p + x$, $p + y$, and $p + x + y + z$ are atoms then~\eqref{eq:quasi-fundamental} gives two distinct factorizations of length two. Indeed, neither $p + x$ nor $p + y$ is associate to $p$, because $x, y \notin pR$. Therefore $R$ is not quasi-$2$-factorial. The issue is to choose $x, y, z$ so that all three sums are atoms.

Under the following quasilocal hypotheses, the atomhood of the sums follows from their images modulo the square of the maximal ideal. Distinctness of the factorizations remains a separate condition.

\begin{prop}\label{prop:quasilocal-shift}
    Let $R$ be a quasilocal domain with maximal ideal $\mathfrak{m}$. Suppose that $p \in \mathfrak{m} \setminus \mathfrak{m}^2$ and that $x, y, z \in \mathfrak{m}^2$ satisfy $xy = pz$. Then $p$, $p + x$, $p + y$, and $p + x + y + z$ are atoms. If the two factorizations in~\eqref{eq:quasi-fundamental} are distinct up to order and associates then $R$ is not quasi-$2$-factorial.
\end{prop}

\begin{proof}
    Since $x, y, z \in \mathfrak{m}^2$, we have
    \[
        p + x \equiv p + y \equiv p + x + y + z \equiv p \pmod{\mathfrak{m}^2}.
    \]
    Thus, all four elements belong to $\mathfrak{m} \setminus \mathfrak{m}^2$. Each is a nonzero nonunit. If one were a product of two nonunits then both factors would lie in $\mathfrak{m}$, and their product would belong to $\mathfrak{m}^2$, a contradiction. Hence all four elements are atoms. The final assertion follows from~\eqref{eq:quasi-fundamental}.
\end{proof}

A localization argument for Question~\ref{question:quasi-two-domain} would require verifying atomicity and quasi-$2$-factoriality in the localized domain. Localization can change which elements are atoms and the lengths of their factorizations, so these properties require justification after localization.

%%%%%%%%%%%%%%%%%%%%%%%%%%%%%%%
\medskip
\subsection{Positive Semirings}

Theorem~\ref{thm:bilf} imposes uniqueness at every factorization length. It is natural to ask whether uniqueness at a single prescribed length suffices.

\begin{question}\label{question:positive-quasi}
    Fix an integer $n \ge 2$. Is $\nn_0$ the only subsemiring $S \subseteq \rr_{\ge 0}$ containing $1$ such that both $(S, +)$ and $S^*$ are quasi-$n$-factorial?
\end{question}

Atomicity of both monoids is part of the hypotheses in Question~\ref{question:positive-quasi}. The following example shows that factoriality of the multiplicative monoid alone does not force the conclusion.

\begin{example}\label{example:dyadic-semiring}
    Consider the positive semiring
    \[
        S := \nn_0[1/2] = \{m/2^k : m, k \in \nn_0\}.
    \]
    Its multiplicative units are the powers $2^j$ with $j \in \zz$, and each nonzero element has a unique expression
    \[
        2^j\prod_{\substack{p\text{ odd}\\p\text{ prime}}}p^{v_p},
    \]
    where $j \in \zz$, each $v_p \in \nn_0$, and only finitely many exponents $v_p$ are nonzero. Consequently, the multiplicative atoms are the unit multiples of odd primes, and $S^*$ is factorial. In particular, it is quasi-$n$-factorial for every integer $n \ge 2$. Nevertheless, $S \ne \nn_0$. The additive monoid has no atoms, since every $s \in S^*$ satisfies
    \[
        s = s/2 + s/2.
    \]
    Thus, this example does not satisfy the additive atomicity required in Question~\ref{question:positive-quasi}. \hfill $\blacksquare$
\end{example}

We finish the paper by classifying the rational and transcendental monogenic positive semirings whose multiplicative monoids are quasi-$n$-factorial for a prescribed integer $n \ge 2$. This will answer Question~\ref{question:positive-quasi} affirmatively for these semirings under the weaker assumption that the additive monoid has an atom.

\begin{prop}\label{prop:monogenic-quasi}
    Fix an integer $n \ge 2$, and let $\alpha \in \rr_{> 0}$ be rational or transcendental. Put $S := \nn_0[\alpha]$. Then $S^*$ is quasi-$n$-factorial if and only if $\alpha \in \nn$ or $\alpha^{-1} \in \nn$. In either case, $S^*$ is factorial.
\end{prop}

\begin{proof}
    We first show that transcendental generators and nonintegral rational generators with numerator greater than one yield two distinct multiplicative factorizations of length two. Consider the following polynomials:
    \[
        f_1 := 1 + x + x^2, \qquad f_2 := 1 + x^3, \qquad f_3 := 1 + x, \qquad f_4 := 1 + x^2 + x^4.
    \]
    Each is an atom of the multiplicative monoid $\nn_0[x] \setminus \{0\}$. Indeed, if $f_i = gh$ with $g, h \in \nn_0[x]$ then $g(1)h(1) = f_i(1)$ is an ordinary prime. One factor, say $g$, satisfies $g(1) = 1$ and is therefore a monomial with coefficient one. Since $f_i(0) = 1$, this monomial must be $1$. The identity
    \begin{equation}\label{eq:quasi-polynomial}
        f_1f_2 = f_3f_4
    \end{equation}
    gives two distinct factorizations of length two in $\nn_0[x] \setminus \{0\}$, whose only unit is $1$. This identity and its failure of unique factorization appear in~\cite[Example 6.1]{GP23}.

    If $\alpha$ is transcendental then evaluation at $\alpha$ is a semiring isomorphism from $\nn_0[x]$ to $S$. Thus,~\eqref{eq:quasi-polynomial} gives two distinct factorizations of length two in $S^*$.

    Suppose next that $\alpha = a/b$ in lowest terms, where $a, b \ge 2$. By~\cite[Corollary 3.8]{DGZ26}, if $f \in \nn_0[x]$ has all coefficients in $\{0, 1\}$ then $f$ is the unique polynomial in $\nn_0[x]$ representing $f(\alpha)$. Applying this to $f = 1$ shows that $S^\times = \{1\}$: if $g(\alpha)h(\alpha) = 1$ with $g, h \in \nn_0[x]$ then $gh = 1$, so $g = h = 1$. Each $f_i(\alpha)$ is an atom, because an equality $f_i(\alpha) = g(\alpha)h(\alpha)$ forces $f_i = gh$, and the atomhood of $f_i$ forces one factor to be $1$. The same uniqueness shows that the four values $f_i(\alpha)$ are pairwise distinct. Hence evaluation of~\eqref{eq:quasi-polynomial} again gives two distinct factorizations of length two in $S^*$.

    In both cases, adjoining $n - 2$ copies of the atom $f_3(\alpha)$ to each factorization yields two distinct factorizations of length $n$. Thus, quasi-$n$-factoriality forces $\alpha \in \nn$ or $\alpha^{-1} \in \nn$.

    It remains to prove factoriality in these two cases. If $\alpha \in \nn$ then $S = \nn_0$. If $\alpha = 1/d$ for an integer $d \ge 2$ then a positive rational belongs to $S$ precisely when every prime divisor of its reduced denominator divides $d$. Its multiplicative units are the positive rationals whose reduced numerators and denominators have no prime divisors outside those dividing $d$. Hence every nonzero element of $S$ has a unique expression as a unit times a product of ordinary primes not dividing $d$, up to the order of those primes. Therefore $S^*$ is factorial and, in particular, quasi-$n$-factorial.
\end{proof}

As a consequence, if $\alpha \in \rr_{> 0}$ is rational or transcendental, $S := \nn_0[\alpha]$, and $S^*$ is quasi-$n$-factorial for some integer $n \ge 2$ then the existence of an atom in $(S, +)$ forces $S = \nn_0$. Indeed, if $\alpha = 1/d$ for an integer $d \ge 2$ then every $s \in S^*$ is the sum of $d$ copies of $s/d \in S^*$, so $(S, +)$ has no atoms. Thus, Question~\ref{question:positive-quasi} has an affirmative answer for rational and transcendental monogenic positive semirings.

The polynomial obstruction in the proof also holds over a broader class of coefficient semidomains. Let $T$ be a semidomain that is not an integral domain. By~\cite[Lemma 6.2]{GP23}, the four polynomials used above remain atoms of $T[x]^*$. They are pairwise nonassociate, since the units of $T[x]$ are constant and all four constant terms are $1$. For each integer $n \ge 2$, adjoining $n - 2$ copies of the same atom to both factorizations in~\eqref{eq:quasi-polynomial} shows that $T[x]^*$ is not quasi-$n$-factorial. Let $K$ be the fraction field of $T - T$. If $\alpha$ lies in a field extension of $K$ and is transcendental over $K$ then evaluation at $\alpha$ gives an isomorphism $T[x] \cong T[\alpha]$, so the same obstruction applies to $T[\alpha]^*$.

%%%%%%%%%%%%%%%%%%%%%%%%%%
%%%%%%%%%%%%%%%%%%%%%%%%%%
\bigskip
\section*{Acknowledgments}

During the period of this collaboration, the authors were part of CrowdMath, a year-long math research program hosted by the MIT Mathematics department and the Art of Problem Solving. The authors thank the program for providing this rewarding research opportunity.

%%%%%%%%%%%%%%%%%%%%%%%%%%
%%%%%%%%%%%%%%%%%%%%%%%%%%

\end{document}